\documentclass{llncs}
\usepackage[english]{babel}
\usepackage{amssymb,amsmath,amsrefs}
\usepackage{bbold}
\usepackage{stackrel}
\usepackage{mathrsfs}
\usepackage{cleveref, enumerate}

\newtheorem{thm}{Theorem}

\newtheorem{cor}[thm]{Corollary}
\newtheorem{defi}[thm]{Definition}

\newtheorem{nota}[thm]{Notation}

\newtheorem{princ}[thm]{Principle}

\newcommand\be{\begin{equation}}
\newcommand\ee{\end{equation}} 

\usepackage{amsmath,amsfonts} 

\usepackage[applemac]{inputenc}

\newcommand{\usftext}[1]{\textsf{\upshape #1}}

\newcommand{\QFAC}{\ensuremath{{\usftext{QF-AC}}}} 
\newcommand{\eefa}{\ensuremath{{\usftext{E-EFA}}}} 

\newcommand{\ACA}{\ensuremath{\usftext{ACA}}} %
\newcommand{\RCA}{\ensuremath{\usftext{RCA}}} %
\newcommand{\SCF}{\ensuremath{\usftext{SCF}}} %

\newcommand{\ZFC}{\ensuremath{\usftext{ZFC}}}
\newcommand{\IST}{\ensuremath{\usftext{IST}}}

\newcommand{\WKL}{\ensuremath{\usftext{WKL}}}

\def\R{{\mathbb{R}}}
\def\({\textup{(}}
\def\){\textup{)}}

\def\st{\textup{st}}
\def\asa{\leftrightarrow}

\def\di{\rightarrow}

\def\LUB{\textup{\textsf{LUB}}}

\def\ns{\textup{\textsf{ns}}}
\def\HBU{\textup{\textsf{HBU}}}

\newbox\gnBoxA
\newdimen\gnCornerHgt
\setbox\gnBoxA=\hbox{$\ulcorner$}
\global\gnCornerHgt=\ht\gnBoxA
\newdimen\gnArgHgt

\def\bdefi{\begin{defi}\rm}
\def\edefi{\end{defi}}
\def\bnota{\begin{nota}\rm}
\def\enota{\end{nota}}
\def\FIVE{\Pi_{1}^{1}\text{-\textup{\textsf{CA}}}_{0}}

\def\ATR{\textup{\textsf{ATR}}}
\def\ZFC{\textup{\textsf{ZFC}}}
\def\IST{\textup{\textsf{IST}}}

\def\STP{\textup{\textsf{STP}}}

\def\ns{\textup{\textsf{ns}}}
\def\RCA{\textup{\textsf{RCA}}}
\def\({\textup{(}}
\def\){\textup{)}}

\def\RCAo{\textup{\textsf{RCA}}_{0}^{\omega}}
\def\WKL{\textup{\textsf{WKL}}}

\def\WWKL{\textup{\textsf{WWKL}}}
\def\bye{\end{document}}
\def\N{{\mathbb  N}}
\def\Q{{\mathbb  Q}}
\def\R{{\mathbb  R}}

\def\B{{\textsf{\textup{B}}}}

\def\st{\textup{st}}
\def\di{\rightarrow}

\def\asa{\leftrightarrow}

\def\ACA{\textup{\textsf{ACA}}}

\def\QFAC{\textup{\textsf{QF-AC}}}

\def\HBU{\textup{\textsf{HBU}}}
\def\HBU{\textup{\textsf{HBU}}}

\def\FF{\textup{\textsf{FF}}}

\def\LMP{\textup{\textsf{LMP}}}
\def\SCF{\textup{\textsf{SCF}}}
\def\MLR{\textup{\textsf{MLR}}}

\def\ml{\textup{\textsf{ml}}}

\def\WCF{\textup{\textsf{WCF}}}
\def\HAC{\textup{\textsf{HAC}}}
\def\INT{\textup{\textsf{int}}}

\newcommand{\HACint}{\ensuremath{\usftext{HAC}_\intern}}
\newcommand{\PFTPA}{\ensuremath{\usftext{PF-TP}_{\forall}}}

\newcommand{\forallst}{\forall^{\st{}}}
\newcommand{\existsst}{\exists^{\st{}}}
\def\FIVE{\Pi_{1}^{1}\text{-\textsf{CA}}_{0}}

\def\intern{\textup{\textsf{int}}}

\numberwithin{equation}{section}
\numberwithin{thm}{section}

\begin{document}
\title{Some nonstandard equivalences in Reverse Mathematics}

\author{Sam Sanders\thanks{This research was supported by the Alexander von Humboldt Foundation and LMU Munich (via the Excellence Initiative).}}
\institute{Center for Advanced Studies \&\\ Munich Center for Mathematical Philosophy, \\LMU Munich, Germany
\email{sasander@me.com}}

%
\maketitle

\begin{abstract}
Reverse Mathematics (RM) is a program in the foundations of mathematics founded by Friedman and developed extensively by Simpson.  The aim of RM is finding the minimal axioms needed to prove a theorem of ordinary (i.e.\ non-set theoretical) mathematics.  
In the majority of cases, one also obtains an \emph{equivalence} between the theorem and its minimal axioms.  This equivalence is established in a weak logical system called \emph{the base theory}; four prominent axioms which boast lots of such equivalences are dubbed 
\emph{mathematically natural} by Simpson.  In this paper, we show that a number of axioms from Nonstandard Analysis are equivalent to theorems of ordinary mathematics \emph{not involving Nonstandard Analysis}.   
These equivalences are proved in a \emph{weak} base theory recently introduced by van den Berg and the author.  
In particular, our base theories have the first-order strength of \emph{elementary function arithmetic}, in contrast to the original version of this paper \cite{samcie}.  
Our results combined with Simpson's criterion for naturalness suggest the controversial point that Nonstandard Analysis is actually mathematically natural.      
\end{abstract}


\section{Introduction}\label{intro}
Reverse Mathematics (RM) is a program in the foundations of mathematics founded by Friedman (\cite{fried}) and developed extensively by Simpson (\cite{simpson2}) and others.  We refer to the latter for an overview of RM and will assume basic familiarity, in particular with the \emph{Big Five} systems of RM.  The latter are (still) claimed to capture the majority of theorems of ordinary (i.e.\ non-set theoretical) mathematics (\cite{montahue}*{p.\ 495}).  Our starting point is the following quote by Simpson on the `mathematical naturalness' of logical systems from \cite{simpson2}*{I.12}:
\begin{quote}
From the above it is clear that the [Big Five] five basic systems $\RCA_{0}$, $\WKL_{0}$, $\ACA_{0}$, $\ATR_{0}$, $\FIVE$ arise naturally from investigations of the Main Question. The proof that these systems are mathematically natural is provided by Reverse Mathematics.
\end{quote}
In a nutshell, according to Simpson, the many equivalences in RM, proved over $\RCA_{0}$ and involving the other four Big Five, imply that the Big Five systems are mathematically natural.   
In this paper, we show that a number of axioms from \emph{Nonstandard Analysis} (NSA) are equivalent to theorems of ordinary mathematics \emph{not involving NSA}.   
These results combined with Simpson's criterion for naturalness suggest the controversial point that NSA is actually mathematically natural.  
Indeed, both Alain Connes and Errett Bishop have expressed extremely negative (but unfounded; see \cite{samsynt}) opinions of NSA, in particular its naturalness. 

\smallskip

Finally, the aforementioned equivalences are proved in a (weak) base theory recently introduced by van den Berg and the author in \cite{bennosam}.  
We sketch the main properties of this base theory in Section \ref{prelim} and prove our main results in Section~\ref{mikeh}.    
The latter include an equivalence between the \emph{Heine-Borel compactness} (for \textbf{any} open cover) of the unit interval and the \emph{nonstandard compactness} of Cantor space.  
We obtain similar results based on $\WWKL$, a weakening of $\WKL$.  The main improvement over \cite{samcie} is that the base theories in this paper are $\Pi_{2}^{0}$-conservative over \emph{elementary function arithmetic} (See \cite{simpson2}*{II.8} for the latter).  

\section{A base theory from Nonstandard Analysis}\label{prelim}
We introduce the system $\B_{0}$ from \cite{bennosam}.
This system is a $\Pi_{2}^{0}$-conservative extension of \textsf{EFA} (aka $I\Delta_{0}+\textsf{EXP}$) enriched with all finite types 
and fragments of Nelson's axioms of \emph{internal set theory} (\cite{wownelly}), a well-known axiomatic approach to NSA.  

\smallskip

Let $\eefa^{\omega}$ be $\textsf{EFA}$ enriched with all finite types, i.e.\ Kohlenbach's system \textsf{E-G$_{3}$A$^{\omega}$} (\cite{kohlenbach3}*{p.\ 55}).
The language of $\B_{0}$ is obtained from that of $\eefa^{\omega}$ by adding unary predicates `$\st^\sigma$' for any finite type $\sigma$. 
Formulas in the old language of $\eefa^{\omega}$, i.e.\ those not containing these new symbols, are \emph{internal}; By contrast, general formulas of $\B_{0}$ are \emph{external}. 
The new `st' predicates give rise to two new quantifiers as in \eqref{quants}, and we omit type superscripts whenever possible.
\be\label{quants}
(\forallst x) \Phi(x)  \equiv  ( \forall x ) (\, \st(x)\rightarrow\Phi(x)  )\textup{ and }( \existsst x) \Phi(x)  \equiv  (\exists x ) (\, \st(x)\wedge\Phi(x) ).
\ee
The system $\B_{0}$ is $\eefa^{\omega}+\QFAC^{1,0}$, plus the basic axioms as in Definition \ref{defke}, and fragments of Nelson's axioms of internal set theory $\IST$, namely \emph{Idealisation}~\textsf{I}, \emph{Standardisation} $\HACint$, and \emph{Transfer} $\PFTPA$, defined as follows. 
\bdefi[$\QFAC$]
For all finite types $\sigma, \tau$ and quantifier-free $A$:
\be\tag{$\QFAC^{\sigma,\tau}$}
(\forall x^{\sigma})(\exists y^{\tau})A(x, y)\di (\exists Y^{\sigma\di \tau})(\forall x^{\sigma})A(x, Y(x))
\ee
\edefi
\begin{defi}\label{defke}~
\begin{enumerate}
\item The axioms $\st(x) \land x = y \to \st(y)$ and $\st(f)\wedge\st(x)\rightarrow\st(fx)$.
\item The axiom $\st(t)$ for each term $t$ in the language of $\B_{0}$.
\item The axiom $\st_{0}(x) \land y \leq_{0} x \to \st_{0}(y)$.
\end{enumerate}
\end{defi}
For the next definition, we note that $x$ in \eqref{horkilua} and $F(x)$ in \eqref{HACINTJE} are both a \emph{finite sequence} of objects of type $\sigma$, as discussed in Notation \ref{klaf}.  

\bdefi[Fragments of $\IST$]
\begin{enumerate}
\item$\HAC_{\INT}$: For any internal formula $\varphi$, we have
\be\label{HACINTJE}
(\forall^{\st}x^{\rho})(\exists^{\st}y^{\sigma})\varphi(x, y)\di \big(\exists^{\st}F^{\rho\di \sigma}\big)(\forall^{\st}x^{\rho})(\exists y^{\sigma}\in F(x))\varphi(x,y).
\ee
\item $\textsf{I}$: For any internal formula $\varphi$, we have
\be\label{horkilua}
(\forall^{\st} x^{\sigma})(\exists y^{\tau} )(\forall z^{\sigma}\in x)\varphi(z,y)\di (\exists y^{\tau})(\forall^{\st} z^{\sigma})\varphi(z,y).
\ee
\item $\PFTPA:$ For any internal $\varphi$ with all parameters shown, we have $ \forallst  x \,  \varphi( x) \to \forall  x \, \varphi( x)$, i.e.\ $x$ is the \textbf{only} free variable in $\varphi(x)$.
\end{enumerate}
\edefi
\begin{nota}[Finite sequences]\label{klaf}\rm
There are at least two ways of approaching `finite sequences of objects of type $\sigma$' in $\eefa^{\omega}$:  
First of all, as in \cite{brie}, we could extend $\eefa^{\omega}$ with types $\sigma^*$ for finite sequences of objects of type $\sigma$, add constants for the empty sequence and the operation of prepending an element to a sequence, as well as a list recursor satisfying the expected equations.  

\smallskip

Secondly, as in \cite{bennosam}, we could exploit the fact that one can code finite sequences of objects of type $\sigma$ as a single object of type $\sigma$ in such a way that every object of type $\sigma$ codes a sequence. Moreover, the operations on sequences, such as extracting their length or concatenating them, are given by terms in G\"odel's $T$. 

\smallskip

We choose the second option here and will often use set-theoretic notation as follows: `$\emptyset$' is (the code of) the empty sequence,` $\cup$' stands for concatenation, and `$\{ x \}$' for the finite sequence of length $1$
with sole component $x$. For $x$ and $y$ of the same type we will write $x \in y$ if $x$ is equal to one of the components of the sequence coded by $y$.  Furthermore, for $\alpha^{0\di \rho}$ and $k^{0}$, the finite sequence $\overline{\alpha}k$ is exactly $\langle \alpha(0), \alpha(1), \dots, \alpha(k-1) \rangle$.  
Finally, if $Y$ is of type $\sigma \to \tau$ and $x$ is of type $\sigma$ we define $Y[x]$ of type $\tau$ as $Y[x] := \cup_{ f \in Y} f(x)$.
\end{nota}
With this notation in place, we can now formulate a crucial theorem from \cite{bennosam}*{\S3}.
\begin{thm}\label{consresult}
For $\varphi$ internal and $\Delta_{\intern}$ a collection of internal formulas, if the system $\B_{0}+\Delta_{\intern}$ proves $ (\forall^{\st}x)(\exists^{\st}y)\varphi(x,y)$, then 
\be\label{keaker}
\eefa^{\omega}+\QFAC^{1,0}+\Delta_{\intern}\vdash (\exists \Phi)(\forall x)(\exists y\in \Phi(x))\varphi(x,y)
 \ee
\end{thm}
By the results in Section \ref{mikeh} and \cite{bennosam}, $\PFTPA$ is useful for obtaining equivalences as in RM.  
However, this `usefulness' comes at a price, as $\B_{0}^{-}$ (i.e.\ $\B_{0}\setminus \PFTPA$) satisfies the following, 
where a \emph{term of G\"odel's system $T$ is obtained}, to be compared to the \emph{existence} of a functional in \eqref{keaker}.
Hence, $\PFTPA$ seems usuitable for proof mining, as the latter deals with extracted \emph{terms} 
\begin{thm}[Term extraction]\label{consresultcor}
If $\Delta_{\textsf{\textup{int}}}$ is a collection of internal formulas and $\psi$ is internal, and $\B_{0}^{-} + \Delta_{\textsf{\textup{int}}} \vdash (\forall^{\st}{x})(\exists^{\st}{y})\psi({x},{y})$, 
then one can extract from the proof a term $t$ from G\"odel's $T$ such that
\be\label{effewachten}
\textup{\eefa$^{\omega}$}+\QFAC^{1,0}+ \Delta_{\textsf{\textup{int}}} \vdash (\forall {x})(\exists {y}\in t({x}))\psi({x},{y}).
\ee
\end{thm}
We finish this section with some notations.  
\begin{nota}[Equality]\label{equ}\rm
The system $\textsf{E-EFA}^{\omega}$ includes equality `$=_{0}$' for numbers as a primitive.  Equality `$=_{\tau}$' for $x^{\tau},y^{\tau}$ is:
\be\label{aparth}
[x=_{\tau}y] \equiv (\forall z_{1}^{\tau_{1}}\dots z_{k}^{\tau_{k}})[xz_{1}\dots z_{k}=_{0}yz_{1}\dots z_{k}],
\ee
if the type $\tau$ is composed as $\tau\equiv(\tau_{1}\di \dots\di \tau_{k}\di 0)$.  Inequality `$\leq_{\tau}$' is \eqref{aparth} with $=_{0}$ replaced by $\leq_{0}$.
Similarly, we define `approximate equality $\approx_{\tau}$' as:
\be\label{aparth2}
[x\approx_{\tau}y] \equiv (\forall^{\st} z_{1}^{\tau_{1}}\dots z_{k}^{\tau_{k}})[xz_{1}\dots z_{k}=_{0}yz_{1}\dots z_{k}]
\ee
\end{nota}
\begin{nota}[Real numbers and related notions in $\B_{0}$]\label{keepintireal}\rm~
\begin{enumerate}
\item Natural numbers correspond to type zero objects.  Rational numbers are defined as quotients of natural numbers, and `$q\in \Q$' has its usual meaning.    
\item A (standard) real number $x$ is a (standard) fast-converging Cauchy sequence $q_{(\cdot)}^{1}$, i.e.\ $(\forall n^{0}, i^{0})(|q_{n}-q_{n+i})|<_{0} \frac{1}{2^{n}})$.  
\item We write `$x\in \R$' to express that $x^{1}=(q^{1}_{(\cdot)})$ is a real as in the previous item and $[x](k):=q_{k}$ for the $k$-th approximation of $x$.    
\item Two reals $x, y$ represented by $q_{(\cdot)}$ and $r_{(\cdot)}$ are \emph{equal}, denoted $x=_{\R}y$, if $(\forall n^{0})(|q_{n}-r_{n}|\leq \frac{1}{2^{n-1}})$. Inequality $<_{\R}$ is defined similarly.         
\item We  write $x\approx y$ if $(\forall^{\st} n^{0})(|q_{n}-r_{n}|\leq \frac{1}{2^{n}})$ and $x\gg y$ if $x>y\wedge x\not\approx y$.  
\item Functions $F:\R\di \R$ are represented by $\Phi^{1\di 1}$ such that 
\be\tag{\textsf{RE}}\label{furg}
(\forall x, y)(x=_{\R}y\di \Phi(x)=_{\R}\Phi(y)).
\ee
\item Sets of natural numbers $X^{1}, Y,^{1} Z^{1}, \dots$ are represented by binary sequences.
\end{enumerate}
\end{nota}

\begin{nota}[Using $\HACint$]\rm
As noted in Notation \ref{klaf}, finite sequences play an important role in $\B_{0}$.   
In particular, $\HACint$ produces a functional which outputs a \emph{finite sequence} of witnesses.  
However, $\HACint$ provides an actual \emph{witnessing functional} assuming (i) $\tau=0$ in $\HACint$ and (ii) the formula $\varphi$ from $\HACint$ is `sufficiently monotone' as in: 
$(\forall^{\st} x^{\sigma},n^{0},m^{0})\big([n\leq_{0}m \wedge\varphi(x,n)] \di \varphi(x,m)\big)$.    
Indeed, in this case one simply defines $G^{\sigma+1}$ by $G(x^{\sigma}):=\max_{i<|F(x)|}F(x)(i)$ which satisfies $(\forall^{\st}x^{\sigma})\varphi(x, G(x))$.  
To save space in proofs, we will sometimes skip the (obvious) step involving the maximum of finite sequences, when applying $\HACint$.  
We assume the same convention for other finite sequences e.g.\ obtained from Theorem \ref{consresultcor}, or the contraposition of idealisation \textsf{I}.  
\end{nota}
 
\section{Reverse Mathematics and Nonstandard Analysis}\label{mikeh}
In Sections \ref{nsac} and \ref{nsac2}, we establish the equivalence between the \emph{nonstandard compactness} of Cantor space and the \emph{Heine-Borel compactness} (for \textbf{any} open cover) of the unit interval.  
The latter essentially predates\footnote{Heine-Borel compactness was studied before 1895 by Cousin (\cite{cousin1}*{p.\ 22}).  The collected works of Pincherle contain a footnote by the editors (\cite{tepelpinch}*{p.\ 67}) stating that the associated \emph{Teorema} (from 1882) corresponds to the Heine-Borel theorem.\label{fottsie}} set theory, and is hence definitely part of `ordinary mathematics' in the sense of RM.  
We establish similar results for theorems based on $\WWKL$ in Section \ref{ferenginar}.  We shall use `computable' in the sense of Kleene's schemes S1-S9 inside $\ZFC$ (\cite{longmann}*{\S5.1.1}).    

\subsection{Nonstandard compactness and the special fan functional}\label{nsac}
The main result of this section is Theorem \ref{massiveBT}, which establishes an equivalence involving the nonstandard compactness of Cantor space and the \emph{special fan functional}, introduced in \cite{samGH} and studied in detail in \cite{dagsam}.   
The variable `$T$' is reserved for trees, and `$T\leq_{1}1$' means that $T$ is a binary tree. 
\bdefi[Special fan functional]\label{dodier}
We define $\SCF(\Theta)$ as follows for $\Theta^{(2\di (0\times 1))}$:
\[
(\forall g^{2}, T^{1}\leq_{1}1)\big[(\forall \alpha \in \Theta(g)(2))  (\overline{\alpha}g(\alpha)\not\in T)
\di(\forall \beta\leq_{1}1)(\exists i\leq \Theta(g)(1))(\overline{\beta}i\not\in T) \big]. 
\]
Any functional $\Theta$ satisfying $\SCF(\Theta)$ is referred to as a \emph{special fan functional}.
\edefi
From a computability theoretic perspective, the main property of $\Theta$ is the selection of $\Theta(g)(2)$ as a finite sequence of binary sequences $\langle f_0 , \dots, f_n\rangle $ such that the neighbourhoods defined from $\overline{f_i}g(f_i)$ for $i\leq n$ form a cover of Cantor space;  almost as a by-product, $\Theta(g)(1)$ can then be chosen to be the maximal value of $g(f_i) + 1$ for $i\leq n$.  No type two functional computes $\Theta$ such that $\SCF(\Theta)$ (\cite{dagsam}), while the following functional can compute $\Theta$ via a term of G\"odel's $T$ (\cite{samflo}). 
\be\tag{$\exists^3$}\label{hah}
(\exists \xi^{3})(\forall Y^{2})\big[  (\exists f^{1})(Y(f)=0)\asa \xi(Y)=0  \big].
\ee
We stress that $g^{2}$ in $\SCF(\Theta)$ may be \emph{discontinuous} and that Kohlenbach has argued for the study of discontinuous functionals in higher-order RM (\cite{kohlenbach2}).  Furthermore, $\RCAo+(\exists \Theta)\SCF(\Theta)$ is conservative over $\WKL_{0}$ (\cite{dagsam, samGH}), and $\Theta$ naturally emerges from Tao's notion of \emph{metastability}, as discussed in \cite{SB, dagsamII, samflo}.  

\smallskip

The special fan functional arose from $\STP$, the \emph{nonstandard compactness} of Cantor space as in \emph{Robinson's theorem} (\cite{loeb1}*{}).  
This fragment of \emph{Standard Part} is also known as the `nonstandard counterpart of weak K\"onig's lemma' (\cite{keisler1}).  
\be\tag{$\STP$}
(\forall \alpha^{1}\leq_{1}1)(\exists^{\st}\beta^{1}\leq_{1}1)(\alpha\approx_{1}\beta),
\ee  
as explained by the equivalence between $\STP$ and \eqref{fanns}, as follows.
\begin{thm}\label{lapdog}
In $\B_{0}^{-}$, $\STP$ is equivalent to the following:
\begin{align}\label{frukkklk}
(\forall^{\st}g^{2})(\exists^{\st}w^{1}\leq_{1}1, k^{0})\big[(\forall T^{1}\leq_{1}1)\big( & (\forall \alpha^{1} \in w)(\overline{\alpha}g(\alpha)\not\in T)\\
&\di(\forall \beta\leq_{1}1)(\exists i\leq k)(\overline{\beta}i\not\in T)\big) \big], \notag
\end{align}  
as well as to the following:
\begin{align}\label{fanns}
(\forall T^{1}\leq_{1}1)\big[(\forall^{\st}n)(\exists \beta)&(|\beta|=n \wedge \beta\in T ) \di (\exists^{\st}\alpha^{1}\leq_{1}1)(\forall^{\st}n)(\overline{\alpha}n\in T)   \big].
\end{align}
Furthermore, $\B_{0}^{-}$ proves $(\exists^{\st}\Theta)\SCF(\Theta)\di \STP$.
\end{thm}
\begin{proof}
A detailed proof may be found in any of the following: \cites{dagsam,samGH, SB}.  In a nutshell, the implication \eqref{frukkklk}$\leftarrow$\eqref{fanns} follows by taking the contraposition of the latter and introducing standard $g^{2}$ in the antecedent of the resulting formula.  
One then uses \emph{Idealisation} \textsf{I} to pull the standard quantifiers to the front and obtains \eqref{frukkklk}.  The other implication follows by pushing the standard quantifiers in the latter back inside.   
For the remaining implication $\STP\di \eqref{fanns}$ (the other one and the final part then being trivial), one uses \emph{overspill} (See \cite{brie}*{\S3}) to obtain a sequence of nonstandard length for a tree $T\leq_{1}1$ satisfying the antecedent of  \eqref{fanns}, and $\STP$ converts this sequence into a standard path in $T$.  \qed
\end{proof}
For the below results, we need the following corollary which expresses the (trivial but important) fact that the type of the universal quantifier in $\STP$ (and equivalent formulations) may be lowered.  We view $\alpha^{0}\leq_{0}1$ as a finite binary sequence; we define $\hat{\alpha}$ to be $\alpha*00\dots$, i.e.\ the type one object obtained by concatenating $\alpha$ with $0^{1}$.  Similarly, $T^{0}\leq_{0}1$ is a binary tree of type zero, and $\SCF_{0}(\Theta)$ is the specification of $\Theta$ restricted to trees $T^{0}\leq_{0}1$.    
\begin{cor}\label{corcorcor}
In $\B_{0}^{-}$, $\STP$ is equivalent to $(\forall \alpha^{0}\leq_{0}1)(\exists^{\st}\beta^{1}\leq1)(\hat{\alpha}\approx_{1}\beta)$, and also to the following:
\begin{align}\label{fanns2}
(\forall T^{0}\leq_{0}1)\big[(\forall^{\st}n)(\exists \beta)&(|\beta|=n \wedge \beta\in T ) \di (\exists^{\st}\alpha^{1}\leq_{1}1)(\forall^{\st}n)(\overline{\alpha}n\in T)   \big], 
\end{align}
and also the following:
\begin{align}\label{frukkklk3}
(\forall^{\st}g^{2})(\exists^{\st}w^{1}\leq_{1}1, k^{0})\big[(\forall T^{0}\leq_{0}1)\big( & (\forall \alpha^{1} \in w)(\overline{\alpha}g(\alpha)\not\in T)\\
&\di(\forall \beta\leq_{1}1)(\exists i\leq k)(\overline{\beta}i\not\in T)\big) \big].   \notag
\end{align}  
The system $\eefa^{\omega}+\QFAC^{1,0}$ proves $(\exists \Theta)\SCF(\Theta)\asa (\exists \Theta_{0})\SCF_{0}(\Theta_{0})$.  
\end{cor}
\begin{proof}
Now, $\eqref{fanns2}\asa \eqref{frukkklk3}$ follows in the same way as for $\eqref{fanns}\asa \eqref{frukkklk}$.
The first forward implication is trivial while the first reverse implication follows by considering $\overline{\alpha}N$ for nonstandard $N^{0}$ and $\alpha^{1}\leq_{1}1$.
The implication $\STP\di \eqref{fanns2}$ follows in the same way as in the proof of the theorem.  Note that $T^{0}\leq_{0}1$ as in the antecedent of \eqref{fanns2} must be nonstandard by the basic axioms in Definition~\ref{defke}.  The implication $\eqref{fanns2}\di \eqref{fanns}$ follows by restricting $T^{1}$ to sequences of some fixed nonstandard length, which yields a type zero object.  The final equivalence follows by applying Theorem \ref{consresultcor} to `$\B_{0}^{-}\vdash \eqref{frukkklk}\asa \eqref{frukkklk3}$'.  
\qed
\end{proof}
The following theorem was proved in \cite{bennosam, samcie} using respectively the \emph{Suslin functional} and \emph{Turing jump functional} $\exists^{2}$, rather than the \emph{much weaker} fan functional $(\FF)$, 
 where `$Y^{2}\in \textsf{cont}$' means that $Y$ is continuous on $\N^{\N}$.
\be\tag{$\textsf{\textup{FF}}$}\label{FF}
(\exists \Phi^{3})(\forall Y^{2}\in \textsf{\textup{cont}})(\forall f, g\leq 1)(\overline{f}\Phi(Y)=\overline{g}\Phi(Y)\di Y(f)=Y(g))
\ee
\be\tag{$\exists^{2}$}
(\exists \varphi^{2})(\forall f^{1})\big[(\exists n)(f(n)=0)\asa \varphi(f)=0  \big].
\ee
Note that the base theory for the equivalence is conservative over $\WKL_{0}^{*}$, i.e.\ the first-order strength is that of elementary function arithmetic.  
Another noteworthy fact is that $\STP$ deals with second-order objects, while $\Theta$ is fourth-order.   
\begin{thm}\label{massiveBT}
The system $\B_{0}+(\FF)+\QFAC^{2,1}$ proves $ \STP\asa (\exists \Theta)\SCF(\Theta)$, while the system $\B_{0}^{-}+(\exists^{3})+\QFAC$ does not.  
\end{thm}
\begin{proof}
The reverse implication is immediate using $\PFTPA$ and Theorem \ref{lapdog}.  
We now prove the forward implication in $\B_{0}+(\FF)+\QFAC^{2,1}$, first additionally assuming $(\exists^{2})$ and then again additionally assuming $\neg(\exists^{2})$.   The law of excluded middle then yields this implication over $\B_{0}+(\FF)+\QFAC^{2,1}$.
Hence, assume $(\exists^{2})$ and note that $\STP$ implies \eqref{frukkklk3} by Corollary \ref{corcorcor}.
Drop the second `st' in \eqref{frukkklk3}, and apply $\PFTPA$ to the resulting formula to obtain
\begin{align}\label{frukkklk2}
(\forall g^{2})(\exists w^{1}\leq_{1}1, k^{0})\big[(\forall T^{0}\leq_{0}1)[ & (\forall \alpha^{1} \in w)(\overline{\alpha}g(\alpha)\not\in T)\\
&\di(\forall \beta\leq_{1}1)(\exists i\leq k)(\overline{\beta}i\not\in T)] \big], \notag
\end{align} 
where the formula in big square brackets is equivalent to a quantifier-free one, thanks to $(\exists^{2})$.
Apply $\QFAC^{2,1}$ to \eqref{frukkklk2} to obtain $\Theta_{0}$ producing $w^{1}, k^{0}$ from $g^{2}$ as in \eqref{frukkklk2}.  Corollary \ref{corcorcor} yield $(\exists \Theta)\SCF(\Theta)$.  
Again for the forward implication, assume $\neg(\exists^{2})$ and note that all functions on $\N^{\N}$ are continuous by \cite{kohlenbach2}*{Cor.\ 3.7}.
Hence, the fan functional $\Phi$ as in $(\FF)$ applies to all functions on $\N^{\N}$, and we may define $\Theta(g)(2)$ as consisting of all $2^{\Phi(g)}$ binary sequences $\sigma*00\dots$ where $|\sigma|=\Phi(g)$ and $\Theta(g)(1):=\Phi(g)$.
The forward implication now follows.  

\smallskip

The non-implication follows from \cite{dagsam}*{Theorem 4.2} as the latter expresses that the special fan functional is not computable in any type two functional.  Indeed, $\STP$ is equivalent to \eqref{frukkklk} by Theorem \ref{lapdog} and applying Theorem \ref{consresultcor} to 
$\B_{0}^{-}+(\exists^{3})+\QFAC+(\exists \Theta)\SCF(\Theta)\vdash \eqref{frukkklk}$, 
one obtains a term $t$ of G\"odel's $T$ such that $\SCF(t)$, which is impossible.  \qed
\end{proof}
%
%

\subsection{Nonstandard compactness and Heine-Borel compactness}\label{nsac2}
We prove an equivalence between $\STP$ and the Heine-Borel theorem \emph{in the general\footnote{The Heine-Borel theorem in RM is restricted to \emph{countable} covers (\cite{simpson2}*{IV.1}).} case}, i.e.\ the statement that any (possibly uncountable) 
open cover of the unit interval has a finite sub-cover.  In particular, any $\Psi:\R\di \R^{+}$ gives rise to a `canonical' open cover $\cup_{x\in [0,1]}I_{x}$ of $[0,1]$ where $I_{x}^{\Psi}\equiv (x-{\Psi(x)}, x+{\Psi(x)})$. 
Hence, the Heine-Borel theorem trivially implies the following statement: 
\be\label{zosimpelistnie}\tag{$\HBU$}\textstyle
(\forall \Psi:\R\di \R^{+})(\exists w^{1}){(\forall x\in [0,1])(\exists y\in w)(x\in I_{y}^{\Psi})}. 
\ee
By Footnote \ref{fottsie}, $\HBU$ is part of ordinary mathematics as it predates set theory.  
Furthermore, $\HBU$ is equivalent to many basic properties of the \emph{gauge integral} (\cite{dagsamIII}).  The latter is an extension of Lebesgue's integral and provides a (direct) formalisation of the 
Feyman path integral.  Note that the following theorem was proved using $(\exists^{2})$ in the base theory in \cite{samcie}. 
\begin{thm}\label{coreBT}
The system $\B_{0}+\QFAC^{2,1}$ proves that $\STP\asa \HBU\asa \HBU^{\st}$, while the system $\B_{0}^{-}+(\exists^3)+\QFAC$ does not.  
\end{thm}
\begin{proof}
For the first reverse implication, we use the same `excluded middle trick' as in the proof of Theorem \ref{massiveBT}.  
Hence, assuming $\neg(\exists^{2})$, all functionals on $\R$ are continuous, and $\STP\di \WKL^{\st}\di\WKL$ (See \cite{bennosam}*{Theorem 3.7}) yields that all functionals on $[0,1]$ are uniformly continuous by \cite{kohlenbach4}*{\S4}, and $\HBU$ is immediate.    
Next, assuming $(\exists^{2})$, note the latter allows us to (uniformly) convert reals into their binary representation (choosing the one with trailing zeros in case of non-uniqueness).  
Hence, any type two functional can be modified to satisfy \eqref{furg} from Notation~\ref{keepintireal} if necessary.  Hence, $\HBU$ immediately generalises to \emph{any} $\Psi^{2}$.   
Now, for the reverse implication, note that $\HBU$ trivially implies
\be\label{leiferik}\textstyle
(\forall \Psi^{2})(\exists w^{1})\underline{(\forall q^{0}\in [0,1])(\exists y\in w)(|q-y|<\frac{1}{\Psi(y)+1})}, 
\ee
where the underlined formula in \eqref{leiferik} may be treated as quantifier-free, due to the presence of $(\exists^{2})$ in the base theory.  
Applying $\QFAC$ to \eqref{leiferik}, we obtain:  
\be\label{leiferik2}\textstyle
(\exists\Phi^{2\di 1})(\forall \Psi^{2}){(\forall q^{0}\in [0,1])(\exists y\in \Phi(\Psi))(|q-y|<\frac{1}{\Psi(y)+1})}, 
\ee
and applying $\PFTPA$ to \eqref{leiferik2} implies that 
\be\label{leiferik3}\textstyle
(\exists^{\st}\Phi^{2\di 1})(\forall \Psi^{2}){(\forall q^{0}\in [0,1])(\exists y\in \Phi(\Psi))(|q-y|<\frac{1}{\Psi(y)+1})}, 
\ee
and we now show that \eqref{leiferik3} implies $\STP$.  Since standard functionals yield standard outputs for standard inputs by Definition \ref{defke}, \eqref{leiferik3} immediately implies
\[\textstyle
(\forall^{\st} \Psi^{2}){(\forall q^{0}\in [0,1])(\exists^{\st} y^{1}\in [0,1])(|q-y|<\frac{1}{\Psi(y)+1})}.  
\]
Now, $(\forall^{\st}\Psi^{2})(\exists^{\st} y\in [0,1])(|q-y|<\frac{1}{\Psi(y)+1})$ implies $(\exists^{\st}y\in [0,1])(q\approx y)$; 
indeed, $(\forall^{\st}y\in [0,1])(q\not\approx y)$ implies $(\forall^{\st}y\in [0,1])(\exists^{\st} k^{0})(|q-y|\geq \frac{1}{k})$, and applying $\HACint$ yields standard $\Xi^{2}$
such that $(\forall^{\st}y\in [0,1])(\exists  k^{0}\in \Xi(y))(|q-y|\geq \frac{1}{k})$.  Defining standard $\Psi_{0}^{2}$ as $\Psi_{0}(y):=\max_{i<|\Xi(y)|}\Xi(y)(i)$, we obtain $(\forall^{\st}y\in [0,1]))(|q-y|\geq \frac{1}{\Psi_{0}(y)+1})$, a contradiction.  
Hence, we have proved $(\forall q^{0}\in [0,1])(\exists^{\st}y\in [0,1])(q\approx y)$, which immediately yields $(\forall x^{1}\in [0,1])(\exists^{\st}y^{1}\in [0,1])(x\approx y)$, as we have $x\approx [x](N)$ for any $x\in [0,1]$ and nonstandard $N^{0}$.  However, every real has a binary expansion in $\RCA_{0}$ (See \cite{polahirst}), and $\B_{0}^{-}$ similarly proves that every (standard) real has a (standard) binary expansion.  A real with non-unique binary expansion can be be summed with an infinitesimal to yield a real with a unique binary expansion.  
Hence, the previous yields that $(\forall \alpha^{1}\leq_{1}1)(\exists^{\st}\beta^{1}\leq_{1}1)(\alpha\approx_{1}\beta)$, which is just $\STP$.  The law of excluded middle as in $(\exists^{2})\vee \neg(\exists^{2})$ now establishes the reverse implication over $\B_{0}+\QFAC^{2,1}$.  

\smallskip

For the first forward direction, $\STP$ implies $(\forall x^{1}\in [0,1])(\exists^{\st}y^{1}\in [0,1])(x\approx y)$ as in the previous paragraph, and we thus have: 
\be\label{leiferikson}\textstyle
{(\forall^{\st}\Psi^{2})(\forall x^{1}\in [0,1])(\exists^{\st} y^{1}\in [0,1])(|x-y|<\frac{1}{\Psi(y)+1})}, 
\ee
Applying \emph{Idealisation} to \eqref{leiferikson}, we obtain
\be\label{leiferikson2}\textstyle
{(\forall^{\st}\Psi^{2})(\exists^{\st}w^{1})(\forall x^{1}\in [0,1])(\exists y\in w)(|x-y|<\frac{1}{\Psi(y)+1})}.
\ee
Dropping the second `st' in \eqref{leiferikson2} and applying $\PFTPA$, we obtain $\HBU$.

\smallskip

For the equivalence $\HBU^{\st}\asa \STP$, the reverse implication follows from the fact that $\STP$ implies \eqref{leiferikson2}.   
For the forward implication, note that $\HBU^{\st}$ implies \eqref{leiferikson2} by taking $w$ provided by $\HBU^{\st}$ and extending this sequence with all $w(i)\pm \Psi(w(i))$ for $i<|w|$. 
However, \eqref{leiferikson2} implies $\STP$ by the previous.  

\smallskip

Finally, the non-implication follows from \cite{dagsam}*{Theorem 4.2} as the latter expresses that the special fan functional is not computable in any type two functional.  Indeed, $\STP$ is equivalent to \eqref{frukkklk} by Theorem~\ref{lapdog}, and apply Theorem~\ref{consresultcor} to 
$\B_{0}^{-}+(\exists^3)+\HBU+\QFAC\vdash \eqref{frukkklk}$, to obtain a term $t$ of G\"odel's $T$ such that $\SCF(t)$, which is impossible, and we are done.  
\qed
\end{proof}
Finally, we consider the least-upper-bound princple from \cite{dagsamIII}*{\S4}.  To this end, a formula $\varphi(x^{1})$ is called \emph{extensional on $\R$} if we have
\be\label{rebel}
(\forall x, y\in \R)(x=_{\R}y\di \varphi(x)\asa \varphi(y)).  
\ee
Note that the same condition is used in RM for defining open sets as in \cite{simpson2}*{II.5.7}.  
\begin{princ}[$\LUB$]
For second-order $\varphi$ \(with any parameters\), if $\varphi(x^{1})$ is extensional on $\R$ and $\varphi(0)\wedge \neg\varphi(1)$, there is a least $y\in [0,1]$ such that $(\forall z \in (y, 1])\neg\varphi(y)$.  
\end{princ}
By $\LUB^{\st}$ we mean $\LUB$ with all quantifiers relative to `st', \emph{including} those pertaining to the parameters in the formula $\varphi^{\st}$, and all quantifiers in \eqref{rebel}.  
\begin{cor}
The system $\B_{0}^{-}$ proves $\LUB^{\st}\di \HBU^{\st}\di\STP$.
\end{cor}
\begin{proof}
The second implication follows by noting that the above proof does not require $\QFAC$ or $\PFTPA$.  The first implication follows from \cite{dagsamIII}*{Thm 4.2}.
\end{proof}
The previous corollary has noteworthy foundational implications: the axiom $\STP$ (and the same for $\LMP$ from Section \ref{ferenginar}) is what is called a `purely nonstandard axiom' (See \cite{pimpson}*{Remark 3.8}).  
Intuitively speaking, such an axiom does not follow from any true second-order sentence relative to the standard world, i.e.\ purely nonstandard axioms do not follow from standard axioms. 
However, $\STP$ does follow from the third-order sentence $\HBU^{\st}$, as well as from a second-order schema with third-order parameters $\LUB^{\st}$.  
Hence, the notion of “purely nonstandard axiom” is extremely dependent on the 
exact formal framework.  
\subsection{Weak compactness and the weak fan functional}\label{ferenginar}
Clearly, $\HBU$ is a generalisation of $\WKL$ from RM.
In this section, we list results similar to Theorems \ref{massiveBT} and \ref{coreBT} for generalisations of $\WWKL$.
The \emph{weak fan functional} $\Lambda$ from \cite{dagsam} arises from the axiom $\WWKL$, as follows:
\be\tag{$\WWKL$}
(\forall T \leq_{1}1)\big[ \mu(T)>_{\R}0\di (\exists \beta\leq_{1}1)(\forall m)(\overline{\beta}m\in T) \big],
\ee
where `$\mu(T)>_{\R}0$' is $(\exists k^{0})(\forall n^{0})\big(\frac{|\{\sigma \in T: |\sigma|=n    \}|}{2^{n}}\geq\frac{1}{k}\big)$.
Although $\WWKL$ is not part of the Big Five, it sports \emph{some} equivalences (\cite{simpson2}*{X.1}).   
The following fragment of \emph{Standard Part} is the nonstandard counterpart of $\WWKL$, as studied in \cite{pimpson}:  
\be\tag{$\LMP$}
(\forall T^{1} \leq_{1}1)\big[ \mu(T)\gg0\di (\exists^{\st} \beta^{1}\leq_{1}1)(\forall^{\st} m^{0})(\overline{\beta}m\in T) \big],
\ee
where `$\mu(T)\gg 0$' is just the formula $[\mu(T)>_{\R}0]^{\st}$.  
Clearly, $\WWKL$ and $\LMP$ are weakened versions of $\WKL$ and $\STP$; the following weaker version of the special fan functional arises from $\LMP$.  
As for the special one, there is \emph{no unique} weak fan functional, i.e.\ it is in principle incorrect to refer to `the' weak fan functional.  
\bdefi[Weak fan functional] \label{fadier}
We define $\WCF(\Lambda)$ for $\Lambda^{(2\di (1\times 1))}$:
\[\textstyle
(\forall k^{0},g^{2}, T^{1}\leq_{1}1)\big[(\forall \alpha \in \Lambda(g,k)(2))  (\overline{\alpha}g(\alpha)\not\in T)
\di   (\exists n\leq\Lambda(g,k)(1) ) (L_{n}(T)\leq\frac{1}{k})\big].
\]
Any $\Lambda$ satisfying $\WCF(\Lambda)$ is referred to as a \emph{weak fan functional}.
\edefi
Now, $\WWKL$ is equivalent to the following statement: \emph{for every $X^{1}$, there is $Y^{1}$ which is Martin-L\"of random relative to $X$}, as proved in \cite{avi1337}*{Theorem 3.1}.
This equivalence is proved in $\RCA_{0}$, and the latter also suffices to e.g.\ define a \emph{universal} Martin-L\"of test $(U_{i}^{X})_{i\in \N}$ (relative to any $X^{1}$).  
The latter has type $0\di 1$ and represents a universal and effective (relative to $X$) null set, i.e.\ a rare event.  Intuitively, $Y$ is (Martin-L\"of) \emph{random} relative to $X$, if $Y$ is not in such a rare event.  
To make this more precise, define `$f^{1}\in [\sigma^{0}]$' as $\overline{f}|\sigma|=_{0}\sigma$ for any finite binary sequence and define $\MLR(Y, X)$ as $(\exists i^{0})(\forall w^{0}\in U_{i}^{X})(Y\not \in [w])$.   

\smallskip

We can now define restrictions of $\STP$ and $\HBU$ to Martin-L\"of random reals.  
%
\be\tag{$\MLR_{\ns}$}
(\forall^{\st} X^{1})(\forall Y^{1})(\exists^{\st} Z^{1})\big( [\MLR(Y,X)]^{\st}\di Z\approx_{1} Y  ).
\ee
Let $\MLR(X, Y, i)$ be $\MLR(X, Y)$ without the leading quantifier.  Now consider 
\be\tag{$\HBU_{\ml}$}
(\forall \Psi^{2}, k^{0}, X^{1})(\exists w^{1})(\forall Y)(\exists Z\in w)( \MLR(Y, X, k)\di Y \in [\overline{Z}\Psi(Z)] ).
\ee
Note that $\HBU_{\ml}$ expresses that the canonical cover $\cup_{f\in 2^{\N}}[\overline{f}\Psi(f)]$ has a finite sub-cover which
covers all reals which are random \emph{and already outside the universal test at level} $U_{k}^{X}$ of the universal test.  
Since $\mu(U_{k})\leq \frac{1}{2^{k}}$, the finite sub-cover need not cover a measure one set in Cantor space.  
The following theorem is proved in the same way as Theorems \ref{massiveBT} and \ref{coreBT}.  
\begin{thm}\label{probab}
The system $\B_{0}+\QFAC^{2,1}$ proves $\LMP\asa \MLR_{\ns}\asa  \HBU_{\ml}$.  Additionally assuming $(\FF)$, we also obtain an equivalence to $ (\exists \Lambda)\WCF(\Lambda)$.  
\end{thm}
Clearly, we may weaken $(\FF)$ in the theorem to a functional only implying $\WWKL$.  

\medskip

Finally, the `st' in the antecedent of $\LMP$ (and $\MLR_{\ns}$) is essential: 
in particular, we show that $\STP$ (and hence $\Theta$) is \emph{robust} in the sense of RM (\cite{montahue}*{p.\ 495}), but $\LMP$ is not. 
Consider the following variations of $\LMP$ and $\STP$. 
\be\tag{$\LMP^{+}$}
(\forall T \leq_{1}1)\big[ \mu(T)>_{\R}0\di (\exists^{\st} \beta\leq_{1}1)(\forall^{\st} m)(\overline{\beta}m\in T) \big],
\ee
\[
(\forall T \leq_{1}1)\big[(\forall n^{0})(\exists \beta^{0})(\beta\in T\wedge |\beta|=n)\di (\exists^{\st} \beta\leq_{1}1)(\forall^{\st} m)(\overline{\beta}m\in T) \big],
\]
where the second one is called `$\STP^{-}$'. 
We have the following theorem.  
\begin{thm}
In $\B_{0}^{-}+\WWKL$, we have $\STP\asa \LMP^{+}\asa \STP^{-}$.
\end{thm}
\begin{proof}
For the first equivalence, we only need to prove $\STP\leftarrow \STP^{-}$, which follows by taking a tree $T\leq_{1}1$ as in the antecedent $\STP$, noting that by overspill it has 
a sequence of nonstandard length, and extending this sequence with $00\dots$ to obtain a tree as in the antecedent of $\STP^{-}$.  Then $\STP^{-}$ yields a standard path in the standard part of the modified tree, which is thus
also in the standard part of the original tree.  
For $\STP\di \LMP^{+}$, apply $\STP$ to the path claimed to exist by $\WWKL$ and note that we obtain $\LMP^{+}$.  
For $\LMP^{+}\di\STP$, fix $f^{1}\leq_{1}1$ and nonstandard $N$.  Define the tree $T\leq_{1}1$ which is $f$ until height $N$, followed by the full binary tree.  
Then $\mu(T)>_{\R}0$ and let standard $g^{1}\leq_{1}1$ be such that $(\forall^{\st}n)(\overline{g}n\in T)$.   By definition, $f\approx_{1}g$ follows, and we are done.  
\qed
\end{proof}

\section*{References}

\bye

\bye 

\bye